\documentclass[12pt]{amsart}

\usepackage[arrow,dvips,matrix,arrow,ps,color,line,curve,frame]{xy}

\usepackage{amssymb}

\advance\textheight1cm

\let\oldlabel=\label
\def\prellabel{\marginparsep=1em
    \def\label##1{\oldlabel{##1}\ifmmode\else\ifinner\else
         \marginpar{{\footnotesize\ \\ \tt
                    ##1}}\fi\fi}}
\advance\headheight1.15pt
\let\:\colon
\let\epsilon\varepsilon

\let\phi=\varphi
\let\theta=\vartheta

\let\Bbb=\mathbb

\def\opn#1#2{\def#1{\operatorname{#2}}}

\opn\gr{gr}
\opn\gp{gp}
\opn\sn{sn}
\opn\n{n}
\opn\k{k}

\opn\inte{int}

\opn\Proj{Proj}
\opn\Supp{Supp}

\opn\chara{char}
\opn\rank{rank}
\opn\Spec{Spec}
\opn\U{U}
\opn\I{Im}
\opn\W{W}

\def\cc{{\mathfrak c}}
\def\ZZ{{\Bbb Z}}
\def\RR{{\Bbb R}}
\def\CC{{\Bbb C}}
\def\NN{{\Bbb N}}
\def\QQ{{\Bbb Q}}
\def\FF{{\bf F}}

\def\TT{{\Bbb T}}
\def\PP{{\Bbb P}}
\def\AA{{\Bbb A}}

\def\cF{{\mathcal F}}
\def\cV{{\mathcal V}}

\opn\End{End}

\opn\ch{ch}
\opn\CH{CH}
\opn\op{op}
\opn\Pic{Pic}

\opn\Vect{Vect}
\opn\Im{Im}

\newtheorem{lemma}{Lemma}[section]
\newtheorem{corollary}[lemma]{Corollary}
\newtheorem{theorem}[lemma]{Theorem}

\theoremstyle{definition}

\newtheorem{remark}[lemma]{Remark}

\newtheorem{conjecture}[lemma]{Conjecture}

\begin{document}

\title[$K$-theory of toric varieties revisited]{$K$-theory of toric varieties revisited}

\author{Joseph Gubeladze}

\thanks{Supported by NSF grant DMS-1000641}


\address{Department of Mathematics, San Francisco
State University, San Francisco, CA 94132, USA}

\email{soso@sfsu.edu}

\begin{abstract}
After surveying higher $K$-theory of toric varieties, we present Totaro's old (c. 1997) unpublished result on expressing the corresponding homotopy theory via singular cohomology. It is a higher analog of the rational Chern character isomorphism for general toric schemes. In the special case of a projective simplicial toric scheme over a regular ring one obtains a rational isomorphism between the homotopy $K$-theory and the direct sum of $m$ copies of the $K$-theory of the ground ring, $m$ being the number of maximal cones in the underlying fan. Apart from its independent interest, in retrospect, Totaro's observations motivated some (old) and complement several other (very recent) results.  We conclude with a conjecture on the nil-groups of affine monoid rings, extending the nilpotence property. The conjecture holds true for $K_0$.
\end{abstract}

\maketitle

\section{$K$-theory of toric varieties: survey}\label{FundNil}

\subsection{Conventions}

%\cite{TTT}, \cite{TTL}, \cite{kripo}, \cite{Tetale}, \.cite{Tderived}, \cite{CHWWtoric}, \cite{Wkhomotopic}, \cite{CLStoric}, \cite{BGspectral}, \cite{Gglobal}, \cite{Gnil}, \cite{Htoric}, %\cite{Webweight}, \cite{ABWweights}, \cite{FSintoric}, \cite{FMSSspheric}, \cite{Soperations}, \cite{Trieste}

All our rings and monoids are commutative. Unless specified otherwise, the monoid operation is written additively.

Our monoid and convex geometry terminology follows \cite{kripo}. In particular, an \emph{affine monoid} is a finitely generated submonoid of a free abelian group. For an affine monoid $M$ its largest subgroup will be denoted by $\U(M)$. An affine monoid $M$ is called (i) \emph{positive} if $\U(M)=0$, (ii) \emph{normal} if $M$ is isomorphic to a monoid of the form $C\cap\ZZ^d$ for a \emph{finite rational cone} $C\subset\RR^d$, and (iii) \emph{seminormal} if, for every element $x\in\gp(M)$, the inclusions $2x,3x\in M$ imply $x\in M$, where $\gp(M)$ is the group of differences of $M$, i.e., $\gp(M)$ is the universal group to which $M$ maps. We say that $M$ is \emph{simplicial} if the cone $\RR_+M\subset\RR\otimes\gp(M)\cong\RR^{\rank\gp(M)}$ is such.

For a functor, defined  on rings, a natural number $c$, and a ring $R$ we denote by $c_*$ the endomorphism, induced by the monoid ring endomorphism  $R[M]\to R[M]$, $m\mapsto m^c$, where the monoid operation is written multiplicatively.

For generalities on toric varieties the reader is referred to \cite[Chap. 10]{kripo} and \cite{CLStoric}.

All fans considered below are assumed to be finite and rational.

Let $\cF$ be a fan. One calls $\cF$ \emph{simplicial} if the cones in $\cF$ are such. The set of maximal cones in $\cF$ is denoted by $\max(\cF)$. The \emph{toric scheme} over a ring $R$, associated with $\cF$, will be denoted by $\cV_R(\cF)$.

For a fan $\cF$ in $\RR^d$, the scheme $\cV_R(\cF)$ is (i) \emph{complete} if and only if $\cF$ is complete, i.e., $\bigcup_{\max(\cF)}\sigma=\RR^d$, (ii) \emph{projective} if and only if $\cF$ is projective, i.e., there is a full dimensional polytope $P$ in the dual space $(\RR^d)^{\text{op}}$ such that the set of duals of the corner cones of $P$ is exactly $\max(\cF)$, and (iii) \emph{quasi-projective} if and only if $\cF$ is a subfan of a projective fan. For a projective fan $\cF$ the mentioned polytope $P$ can be chosen to have vertices in the dual lattice $(\ZZ^d)^{\text{op}}$ and then one writes $\cV_R(\cF)=\Proj(R[P])$.

We have $\AA^d_R=\Spec(R[\ZZ_+^d])$, $\TT^d_R=\Spec(R[\ZZ^d])$, and $\PP^d_R=\Proj(R[\ZZ_+^{d+1}])=\Proj(R[\Delta_d])=\Proj(R[c\Delta_d])$, where $\Delta_d$ is the unit $d$-simplex, $c$ is any natural number, and $c\Delta_n$ is the $c$-th dilation of $\Delta_d$.

For a cone $\sigma\subset\RR^d$, we denote by $M(\sigma)$ the intersection of the maximal linear subspace of the dual cone $\sigma^{\op}$ with the dual lattice $(\ZZ^d)^{\op}$.

For a sequence of natural numbers $\cc=(c_1,c_2,\ldots)$ and a monoid $M$, we put $M^{\cc}=\lim_{\to}\xymatrix{(M\ar[r]^{c_1\cdot}&M\ar[r]^{c_2\cdot}&M\ar[r]^{c_3\cdot}&\cdots)}$. Thus, for a constant sequence $\cc=(c,c,\ldots)$ we have $\ZZ^{\cc}\cong\ZZ[1/c]$ and, more generally, $G^{\cc}\cong G\otimes \ZZ[1/c]$, where $G$ is any abelian group.

If $X$ is an abelian group, a homomorphism of abelian groups, or a spectral sequence, we will use $X_\QQ$ for $X\otimes\QQ$.

\subsection{Affine monoid rings.} Let $R$ be a regular ring. Prototypes for the results discussed below are the classical isomorphisms of Grothendieck $K_0(R)\cong K_0(\AA_R^{d})\cong K_0(\TT_R^{d})$ and Berthelot $K_0(\PP^d_R)\cong K_0(R)^{d+1}$. Both results were extended to all higher groups by Quillen \cite{Qhk}. From this point on all natural isomorphisms will be written and referred to as equalities.

The following isomorphism of graded rings results from the iterative use of the \emph{Fundamental Theorem of $K$-theory}:
\begin{equation}\label{fundamental}
K_*(\TT^d_R)=K_*(R)\otimes\Lambda^*(\ZZ^d),
\end{equation}
where $\Lambda^1(\ZZ^d)=\ZZ^d\subset\U(R[\ZZ^d])\subset K_1(R[\ZZ^d])$.

For general (even normal) affine monoid rings over regular ring, the equality $K_*(R)=K_*(R[M])$ fails badly (\cite{Gnontrivial}) and the following \emph{nilpotence} result serves as a substitute. Let $\cc=(c_1,c_2,\ldots)$ be a sequence of natural numbers with $c_i\ge2$ for all $i$, $M$ an affine positive monoid, and $R$ a regular ring. Then
\begin{equation}\label{c-nilpotence}
K_*(R[M^{\cc}])= K_*(R)\ \text{if either}\ R\ \text{contains a field or}\ M\ \text{is simplicial}.
\end{equation}
The equality (\ref{c-nilpotence}) in the special case when $\QQ\subset R$ is proved in \cite{Gnil,Gglobal}. Using their $cdh$ techniques, Corti\~nas, Haesemeyer, Walker, and Weibel later gave a shorter proof when $R$ is a field of characteristic $0$ \cite{CHWWtoric}. Both approaches use Corti\~nas' verification of the \emph{KABI} conjecture. In the recent work \cite{CHWWpos}, the same authors settle the general case when $R$ contains a field. When $M$ is simplicial and $R$ is an arbitrary regular ring the equality is proved in \cite[Thm. 6.4]{Ghktt}.

For a constant sequence $\cc=(c,c\ldots)$ with $c\ge2$, an affine monoid $M$, and a regular ring $R$, such that either $M$ is simplicial or $R$ contains a field, (\ref{fundamental}) and (\ref{c-nilpotence}) imply that
$$
K_*(R[M^{\cc}])= K_*(R[\U(M)^{\cc}])=K_*(R)\otimes\Lambda^*(\ZZ[1/c]^d),\quad d=\rank\U(M).
$$

\medskip For $K_1$ and $K_2$, the equality (\ref{c-nilpotence}) can be extended to all regular coefficient rings and all affine monoids. Moreover, the result can be strengthened to the corresponding unstable $K$-groups \cite{GK1,GK2}.

For the non-positive $K$-theory of monoid rings the following stronger equalities hold true. Let $R$ be a regular ring and $M$ an affine monoid. Then:

\begin{equation}\label{K0}
\begin{aligned}
&K_0(R)=K_0(R[M])\ \ \text{for}\ M\ \text{seminormal}\ \ \text{(\cite[Thm. 8.37(a)]{kripo})},\\
&SK_0(R)=SK_0(R[M]),\ \ \text{(\cite[Thm. 8.37(b)]{kripo})},\\
&K_n(R[M])=0,\quad n<0.
\end{aligned}
\end{equation}

\noindent The triviality of the negative groups in the special case of normal monoids directly follows from the first equality in (\ref{K0}) and the fact that $K_n(R[M])$ is a direct summand of $K_0(R[M\times\ZZ^{-n}])$ for $n<0$, a consequence of the Fundamental Theorem. The triviality in the general case of affine monoids is a little more intricate, but the argument is exactly parallel to the proof of \cite[Thm. 8.37(b)]{kripo}.

In the special case when the Krull dimension of $R$ is 1, the stronger unstable versions of the first two equalities in (\ref{K0}) are shown in \cite{Guan,Swg}: if $M$ is seminormal then all projective $R[M]$-modules are extended from the Dedekind ring $R$, while for any affine monoid $M$ all projective $R[M]$-modules are of the form  \emph{free $\oplus$ rank 1}.

\subsection{Toric varieties.} First we assume $\cV=\cV_k(\cF)$ is a complete and smooth toric scheme over a ground field $k$ (equivalently, $\cF$ is complete and the cones in $\cF$ are unimodular). By the Grothendieck-Riemann-Roch theorem, the Chern character yields the ring homomorphism $\ch:K_0(\cV)_\QQ\to\CH^*(\cV)_\QQ$. Jurkiewicz and Danilov determined the target Chow ring long ago.\footnote{For the unreferenced results in this subsection see \cite[Chap. 10]{kripo} and the references therein.} As an additive group, $\CH^*(\cV)=\ZZ^n$, $n=\#\max(\cF)$. It is also known that $K_0(\cV)$ is a free abelian group (Merkurjev-Panin, Morelli, Vezzosi-Vistoli). As for higher groups, Vezzosi-Vistoli \cite{VVhak} have shown the ring isomorphism
$K_*(k)\otimes K_0(\cV)=K_*(\cV)$. So, additively,
\begin{equation}\label{VVhaK}
K_*(\cV)=K_*(k)^m,\quad m=\#\max(\cF).
\end{equation}

For a general simplicial projective toric variety $\cV$ over a large ground field (such as $\CC$), $K_0(\cV)$ may contain a nonzero continuous contribution from the higher nil-groups \cite{Ghuge} (see also \cite{CHWWtoric}). In Section \ref{Conjecture} we offer a conjectural glimpse into the structure of such a continuous parts.

The nilpotence property (\ref{c-nilpotence}) extends to general toric varieties as follows. For $\cc=(c_1,c_2,\ldots)$, $c_i\ge2$, a fan $\cF$, and a regular ring $R$, such that  either $R$ contains a field or $\cF$ is simplicial, one has $K_*(\cV)^{\cc}=K_*(\cV\times\AA^1)^{\cc}$ (\cite[Prop. 4.7]{Ghktt}). Equivalently, $K_*(\cV)^{\cc}=KH_*(\cV)^{\cc}$ (\cite[Thm. 6.9]{CHWWtoric}).

\begin{remark}\label{remark}
(a) An affine positive monoid $M$ admits a \emph{grading}, i.e., a partition $M=\{0\}\cup M_1\cup M_2\cup\cdots$ with $M_i+M_j\subset M_{i+j}$ (\cite[Prop. 2.17]{kripo}). Consequently, for a field $k$ of characteristic 0, Stienstra's operations \cite{Sop} of the big Witt vectors $\text{W}(k)$ on the nil-groups, extended to the graded case by Weibel \cite{Woperations}, make $K_*(k[M])/K_*(k)$ into a $k$-vector space in such a way that the endomorphism
$$
c_*:K_*(k[M])/K_*(k)\to K_*(k[M])/K_*(k)
$$
has no non-zero eigenvalues for $c\ge2$. Totaro uses this observation to conclude that the nil part of the spectral sequence for the $K$-theory of a toric variety, associated with the standard open cover, does not interfere with the homotopy $K$-theory part (see Lemma \ref{Esplits} below). The equalities $(\ref{c-nilpotence})$ make the analysis of the spectral sequence straightforward. Back in the 1990s, Totaro's observation suggested the crucial idea of using the $\text{W}(k)$-action to prove the equalities (\ref{c-nilpotence}) for characteristic $0$ fields -- a project, stalled at that time on the level of classical $K$-groups, and later completed in \cite{Gnil}. The full potential of the Stienstra operations in the study of $K$-theory of monoid rings does not seem to be exhausted, though; see Section \ref{Conjecture}.

\medskip(b) For any projective simplicial toric scheme, Corollary \ref{totaroequalities}(c) below gives an analog of (\ref{VVhaK}) for the Weibel homotopy theory: $KH_*(\cV_R(\cF))_{\QQ}=K_*(R)_{\QQ}^m$, where $R$ is a regular ring, $\cF$ is a projective simplicial fan, and $m=\#\max(\cF)$. For a complete simplicial fan $\cF$ and a field $k$, partial results, suggesting that $KH_*(\cV_k(\cF))_{\QQ}$ depends only on the combinatorial type of $\cF$, are obtained in Massey's recent preprint \cite{Massey}.

H\"uttemann  \cite{Htoric} recently showed that $K_*(\Proj(R[P])=K_*(R)^{n_P+1}\oplus\ ?$ for any lattice polytope $P$ and any ring $R$, where $n_P$ is the least non-negative integer for which the dilated polytope $(n_P+1)P$ has an interior lattice point. Obviously, $n_P\le\dim P$. It is known that $n_P$ is the number of distinct integral roots of the \emph{Ehrhart polynomial} of $P$. Corollary \ref{totaroequalities} suggests that, for a general lattice polytope $P$, the number of copies of $K_*(R)$, splitting off from $K_*(\Proj(R[P])$ rationally, is at least $\#\text{vert}(P)$.
\end{remark}

\section{Spectral sequences, weights, and degeneration}\label{Spectral}

This section is based on Totaro's notes. We also include background material and work over general regular ground rings, instead of Totaro's use of characteristic $0$ fields.

\subsection{Mayer-Vietoris spectral sequence}\label{Cohomological}

Thomason \cite[Section 8]{Tderived} established $K$-theoretical Mayer-Vietoris long exact sequence for covers by two open subschemes. For covers by more than two open subschemes one obtains strongly convergent spectral sequences \cite[Section 8]{Tderived}. The theory, considered in \cite{Tderived}, is the Waldhausen theory, associated with the corresponding category of perfect complexes. The latter coincides with Quillen's $K$-theory for schemes with ample systems of line bundles. This class of schemes includes quasi-projective schemes over affine schemes \cite[Section 3]{Tderived}. The analogous spectral sequence for $KH_*$ had been constructed before in \cite{Wkhomotopic}. In particular, for a finite open cover $\cV=\bigcup_{i=1}^nU_i$, one has
$$
E_1^{p,q}=\bigoplus_{i_0<\cdots<i_{p}}\k_q(U_{i_0}\cap\cdots\cap U_{i_{p}})\Longrightarrow\k_{q-p}(\cV),
$$
where $\k_*$ denotes $K_*$ or $KH_*$ and the indexing is that of Bousfield-Kan (as in \cite{Tderived}):
$$d_r:E_r^{p,q}\to E_r^{p+r,q+r-1}
$$
The indexing used in \cite{Wkhomotopic} is the traditional one, relating to the one above via $q\leftrightarrow-q$.

\medskip\noindent\emph{Convention}. For the rest of Section \ref{Spectral} we let $\cF$ be a quasi-projetive fan in $\RR^d$, $\max(\cF)=\{\sigma_1,\ldots,\sigma_n\}$, and $R$ be a regular ring. Also, we denote by ${}'E$ and ${}''E$, respectively, the $KH_*$- and $K_*$-spectral sequences for the standard open cover of $\cV_R(\cF)$ by affine toric schemes.

\medskip The affine open subscheme $U_{\sigma_i}$, corresponding to $\sigma_i$, is of the form
$$
\Spec(R[\sigma^{\op}_i\cap(\ZZ^d)^{\op}]=\Spec(R[M(\sigma_i)\times M_i])
$$
for some affine normal positive monoid $M_i$. Because $M_i$ admits a grading, we have $KH_*(U_{\sigma_i})=KH_*(R[M(\sigma_i)])=K_*(R[M(\sigma_i)])$. So, in view of the third equality in (\ref{K0}), both ${}'E$ and ${}''E$ are first quadrant spectral sequences.  By (\ref{fundamental}), the first page of ${}'E$ is
\begin{equation}\label{toric Mayer-Vietoris}
\tiny{\xymatrix{
\cdots&&&\\
\bigoplus_{i_0}\bigoplus_{j=0}^2(\Lambda^{2-j}M(\sigma_{i_0})\otimes K_j(R))\ar[r]&\bigoplus_{i_0<i_1}\bigoplus_{j=0}^2(\Lambda^{2-j}M(\sigma_{i_0}\cap\sigma_{i_1})\otimes K_j(R))\ar[r]&\cdots\\
\bigoplus_{i_0}(M(\sigma_{i_0})\bigoplus K_1(R))\ar[r]&\bigoplus_{i_0<i_1}(M(\sigma_{i_0}\cap\sigma_{i_1})\bigoplus K_1(R))\ar[r]&\cdots\\
\oplus_{i_0}K_0(R)\ar[r]&\oplus_{i_0<i_1}K_0(R)\ar[r]&\cdots\\
}}
\end{equation}
That is,
\begin{equation}\label{Epq}
{'E}_1^{p,q}=\bigoplus_{i_0<\cdots<i_{p}}\bigoplus_{j=0}^q\left(\Lambda^{q-j}M(\sigma_{i_0}\cap\cdots\cap\sigma_{i_{p}})\otimes K_j(R)\right).
\end{equation}
We also have
$$
{''E}_1^{p,q}=\left[\bigoplus_{i_0<\cdots<i_{p}}\bigoplus_{j=0}^q\left(\Lambda^{q-j}M(\sigma_{i_0}\cap\cdots\cap\sigma_{i_{p}})\otimes K_j(R)\right)\right]\oplus N^{p,q}
$$
for certain nil-groups $N^{p,q}$.

For any natural number $c\ge2$:
\begin{enumerate}
\item[(i)]
$c_*$ acts on ${}'E$ as well as on ${}''E$,
\item[(ii)]
$c_*$ is the identity on $K_*(R)$,
\item [(iii)]
$c_*$ acts on $\Lambda^aM(\sigma_{i_0}\cap\cdots\cap\sigma_{i_{p}})$ and, therefore, on $\Lambda^aM(\sigma_{i_0}\cap\cdots\cap\sigma_{i_{p}})\otimes K_j(R)$ by multiplication on $c^a$.
\end{enumerate}

The weights, resulting from the action of $c_*$ on ${}'E$ will be referred to as \emph{$\NN$-weights} (because $c\mapsto c_*$ gives rise to an action of the multiplicative monoid $\NN$). Thus the $\NN$-weight of $\Lambda^aM(\sigma_{i_0}\cap\cdots\cap\sigma_{i_{p}})\otimes K_j(R)$ is $c^a$.

The first step in Totaro's analysis of the spectral sequence ${}'E$ is the following

\begin{lemma}\label{Esplits}
Assume either $R$ contains a field or $\cF$ is a simplicial fan. Then ${}'E$ and the $N^{p,q}$ form two spectral subsequences of ${}''E$ and, denoting by $N$ the latter spectral sequence, we have ${}''E={}'E\oplus N$.
\end{lemma}

\begin{proof}
In view of (\ref{c-nilpotence}) and the sublemma below, the properties (i-iii) imply that any homomorphism between subquotients of the groups in (\ref{toric Mayer-Vietoris}), viewed as groups in the first page of ${}''E$, and subquotients of the $N^{p,q}$ becomes zero homomorphism after tensoring with $\ZZ[1/c]$. Since this holds true also for $c'\ge2$, coprime with $c$, there are no non-zero homomorphisms in ${}''E$ of the mentioned type.
\end{proof}

\medskip\noindent{\bf Sublemma.}
\emph{Let $A$ and $B$ be $\ZZ[1/c]$-modules. Assume $\alpha:A\to A$ is multiplication by $c^k$ for some $k\in\NN$ and $\beta:B\to B$ is an endomorphism, such that for any $x\in B$ there exists $l\in\NN$ with $\beta^l(x)=0$. Then in any commutative squares of $\ZZ[1/c]$-modules
$$
\xymatrix{A\ar[r]^f\ar[d]_{\alpha}&B\ar[d]^{\beta}\\
A\ar[r]_f&B
}
\quad \text{and}\quad \xymatrix{B\ar[r]^g\ar[d]_{\beta}&A\ar[d]^{\alpha}\\
B\ar[r]_g&A
}
$$
one has $f=0$ and $g=0$.}

\begin{proof}
Let $x\in A$ and $y\in B$. Assume $\beta^l(f(x))=0$ and $\beta^n(y)=0$ for some $l,n\in\NN$. Then we have $0=(\beta^lf)(x)=(f\alpha^l)(x)=c^{kl}f(x)$ and $0=(g\beta^n)(x)=(\alpha^ng)(y)=c^{kn}g(y)$, forcing $f(x)=0$ and $g(y)=0$.
\end{proof}

Lemma \ref{Esplits} implies that the map $K_*(\cV_R(\cF))\to KH_*(\cV_R(\cF))$ is surjective. Actually, by \cite[Prop. 5.6]{CHWWtoric}, the map $K_*(\cV_R(\cF))\to KH_*(\cV_R(\cF))$ is always split surjective:
\begin{equation}\label{splitsurijective}
K_*(\cV_R(\cF))=KH_*(\cV_R(\cF))\oplus N_*.
\end{equation}
In fact, the corresponding splittings on the open subschemes
$$
U_{\sigma_{i_0}\cap\cdots\cap\sigma_{i_{p}}}=U_{\sigma_{i_0}}\cap\cdots\cap U_{\sigma_{i_{p}}}
$$
are compatible. So the splitting extends to the whole scheme because $KH_*(\cV_R(\cF))$ and $K_*(\cV_R(\cF))$ are the homotopy groups of the homotopy limits of the corresponding \v Cech cosimplicial spectra (\cite[Section 8]{Tderived}, \cite{Wkhomotopic}).

\medskip\noindent\emph{Adams weights.} Adams operations in algebraic $K$-theory with expected properties have been defined with various levels of generality by several people. This includes operations in the affine and regular cases, or with rational coefficients. Since these works usually consider Quillen's theory and, simultaneously, utilize the Mayer-Vietoris property, one usually requires the existence of an ample family of line bundles for the scheme in question (e.g., quasi-projectivity). It seems plausible that the operations can be defined for Thomason's theory in the appropriate generality.

For a commutative ring $R$, Adams operations $\psi^k:K_*(R)\to K_*(R)$, $k\not=0$, are ring homomorphisms. %\cite[**]{loday}).
Let $R$ be a Noetherian ring of Krull dimension $d$ and $q\in\ZZ_+$. Denote by $K_q(R)^{(i)}\subset K_q(R)$ the subgroup of elements of \emph{Adams weight $i$}, i.e., the elements $x$  such that $\psi^k(x)=k^ix$ for all $k$. Then one has:
\begin{equation}\label{Soule}
\begin{aligned}
&K_q(R)\otimes\ZZ\left[1/(q+d-1)!\right]=\bigoplus_{i=2}^{q+d}\left(K_q(R)^{(i)}\otimes\ZZ\left[1/(q+d-1)!\right]\right)\quad\text{for}\ q\ge2,\\
&K_1(R)\otimes\ZZ\left[1/(d+1)!\right]=\bigoplus_{i=1}^{d+1}\left(K_1(R)^{(i)}\otimes\ZZ\left[1/(d+1)!\right]\right),\\
&\U(R)=\U(R)^{(1)},\ \text({i.e.,}\ \psi^k(u)=u^k\ \text{for all}\ u\in\U(R)).
\end{aligned}
\end{equation}
The first two equalities are proved in \cite[\S2.8]{Soperations} and the third is proved in \cite[Cor. 6.8]{Kratzer}.

\medskip\noindent\emph{Notice.} The original formulation in \cite{Soperations} uses the \emph{stable range} of $R$, which is known to be bounded above by $d+1$ \cite[Ch. 5, Thm. 3.5]{Bass}.

\begin{lemma}\label{Adamsoperations}
Assume either $R$ contains a field or $\cF$ is simplicial.
\begin{itemize}
\item[(a)] Adams operations are defined for the whole spectral sequence  ${}'E_\QQ$, i.e., the $\psi^k$ also act on the $({'E}^{p,q}_r)_\QQ$ and the action commutes with the differentials.
\item[(b)] ${}'E$ is a module over $K_*(R)$, i.e., multiplication by elements from $K_*(R)$ commute with the differentials in ${}'E$:
$$
z\cdot d_r^{p,q}(-)=d^{p,q+t}_r(z\cdot-):{'E}^{p,q}_r\to{'E}^{p+r,q+r+t-1}_r,\qquad z\in K_t(R).
$$
\end{itemize}
\end{lemma}

\begin{proof} (a) Adams' operations are defined for the rational $K$-theory of any quasi-projective scheme \cite{Lecomte}. In fact, the operations are defined on the level of the corresponding $K$-theoretical spaces and these operations are functorial w.r.t. morphisms of schemes \cite[Prop. 4.1.2]{Lecomte}. This implies that the operations are natural w.r.t. the homotopy fibers of the maps of $K$-theoretical spaces, induced by morphisms of the underlying schemes. In particular, Adams operations are natural w.r.t. the Mayer-Vietoris long exact sequences, associated to covers by two open subschemes. Then the naturality can be promoted to the Mayer-Vietoris spectral sequence of an open cover along the lines of proof of \cite[Prop. 8.3]{Tderived}. That is, for an open cover of any quasi-projective scheme, Adams' operations are defined on the open sub-schemes and they commute with the differentials in the associated spectral sequence. (In the induction w.r.t. the number of open subschemes involved, following the outline in \cite[Prop. 8.3]{Tderived}, it is important that the operations are defined for quasi-projective -- not just affine -- schemes.)

In particular, the analogue of Lemma \ref{Adamsoperations}(a) holds for the spectral sequence ${}''E_\QQ$. By Lemma \ref{Esplits}, ${}'E_\QQ$ is a subsequence of the latter. (This is where the condition `$R$ contains contains a field or $\cF$ is simplicial' enters the argument.) So, using the multiplicative property of the maps $\psi^k$ and the fact that the $\psi^k$ respect the groups of units (the 3rd equality in (\ref{Soule})), the operations on ${}''E_\QQ$ restrict to ${}'E_\QQ$.

(b) The product in $K$-theory of schemes is functorial w.r.t. the underlying $K$-theoretical spaces (\cite[p. 342]{WAprod}). So the approach used in the proof of \cite[Prop. 8.3]{Tderived} works here, too.
\end{proof}

\medskip\noindent\emph{Notice.} As the referee pointed out, in the special case when $R$ contains a field of characteristic 0, one can define Adams operations directly on the $KH$-spectral sequence without invoking the Adams operations on the $K$-spectral sequence. In fact, this applies to \emph{any} scheme $X$ over a field of characteristic 0, not just toric varieties. In more detail, Haesemeyer \cite{HAdesc} proved that $KH_m(X)$ is the cdh-hypercohomology $H^{-m}_{\text{cdh}}(X,K)$ and, therefore, the general approach, developed by Gillet-Soule \cite{GSfiltr}, shows that $KH_m(X)$ is a $\lambda$-module for $m\ge0$. Alternatively, one can use the approaches, developed by Cisinski \cite{CIdesc} and Riou \cite{RIa1}.

\subsection{Singular cohomology}\label{Singular}
In this subsection the ground ring is $\CC$. For any finitely generated graded $\CC$-algebra $B=A\oplus A_1\oplus A_2\oplus\cdots$ the one-parameter family of maps $B\to B$, $(a,a_1,a_2,\ldots)\to(a,ta_1,t^2a_2,\ldots)$, $t\in[0,1]$, shows that $\Spec(A)_{\text{top}}$ is a strong deformation retract of $\Spec(B)_{\text{top}}$. Here `$-_{\text{top}}$' refers to the underlying complex space. Any affine toric variety over $\CC$ is of the form $\Spec(\CC[M\times\ZZ^n])_{\text{top}}$ for an affine normal positive monoid $M$. Because there is a grading $\CC[M\times\ZZ^n]=\CC[\ZZ^n]\oplus A_1\oplus\cdots$, $\Spec(\CC[M\times\ZZ^n])_{\text{top}}$ is homotopic to the complex torus $(\CC^*)^n$ and, therefore, to the real torus $(S^1)^n$.

Let $\cF$ be a fan, $\max(\cF)=\{\sigma_1,\ldots,\sigma_n\}$, $\cV=\cV_{\text{top}}(\cF)$, and $M(\sigma)$ and $U_{\sigma}$ be as in Section \ref{Cohomological}.

Because $U_{\sigma_{i_0}\cap\cdots\cap\sigma_{i_p}}$ is homotopic to $(S^1)^r$ with $r=\rank M(\sigma_{i_0}\cap\cdots\cap\sigma_{i_p})$, we have $H^q(U_{\sigma_{i_0}\cap\cdots\cap\sigma_{i_p}},\ZZ)=\Lambda^qM(\sigma_{i_0}\cap\cdots\cap\sigma_{i_p})$. Let $\tilde E$ be the integral cohomology spectral sequence of the standard open cover $\cV=\bigcup U_{\sigma_i}$. Then the first page of $\tilde E$ is

\begin{equation}\label{singcoh}
\xymatrix{
\oplus_{i_0}\Lambda^2M(\sigma_{i_0})\ar[r]&\oplus_{i_0<i_1}\Lambda^2M(\sigma_{i_0}\cap\sigma_{i_1})\ar[r]&\oplus_{i_0<i_1<i_2}\Lambda^2M(\sigma_{i_0}\cap\sigma_{i_1}\cap\sigma_{i_2})
\ar[r]&\cdots\\
\oplus_{i_0}M(\sigma_{i_0})\ar[r]&\oplus_{i_0<i_1}M(\sigma_{i_0}\cap\sigma_{i_1})\ar[r]&\oplus_{i_0<i_1<i_2}M(\sigma_{i_0}\cap\sigma_{i_1}\cap\sigma_{i_2})\ar[r]&\cdots\\
\oplus_{i_0}\ZZ\ar[r]&\oplus_{i_0<i_1}\ZZ\ar[r]&\oplus_{i_0<i_1<i_2}\ZZ\ar[r]&\cdots\\
}
\end{equation}
where the differentials are $d_r:\tilde E_r^{p,q}\to\tilde E_r^{p+r,q-r+1}$, i.e.,

$$
\tilde E_1^{p,q}=\oplus_{i_0<\cdots<i_p}\Lambda^qM(\sigma_{i_0}\cap\cdots\cap\sigma_{i_p})\Longrightarrow H^{p+q}(\cV,\ZZ)
$$

\medskip\noindent\emph{Notice.} In practice, it is difficult to keep track of the many cones in $\cF$. To circumvent the difficulty, one uses another more economical spectral sequence -- the spectral sequence resulting from filtering $\cV$ by the unions of the closures of the torus orbits \cite[Chap. 12]{CLStoric}.

\medskip As usual, the action of a natural number $c$ on the spectral sequence $\tilde E$ will be denoted by $c_*$. We see that the rational cohomology $H^q(\cV,\QQ)$ is filtered
$$
0\subset V_0\subset\cdots\subset V_q=H^q(\cV,\QQ)
$$
in such a way that $c_*$ acts on $V_j/V_{j-1}$ by multiplication on $c^j$.

When $\cV$ is complete, the mentioned filtration coincides with the \emph{Deligne weight filtration}. More precisely, one has (\cite[Prop. 1.4]{ABWweights}\cite[\S6]{TTL}):
\begin{align*}
W_0H^q(\cV,\QQ)&\subset W_1H^q(\cV,\QQ)\subset\cdots\\
&\subset W_{2q-1}H^q(\cV,\QQ)\subset W_{2q}H^q(\cV,\QQ)=H^q(\cV,\QQ),
\end{align*}
where, as a special feature of toric varieties,
\begin{align*}
W_{2j}H^q(\cV,\QQ)=W_{2j+1}H^q(\cV,\QQ)=\{z\in H^q(\cV,\QQ)\ |\ c_*(Z)=c^iz\ &\text{for some}\ i\le j\},\\
&j=0,\ldots,q-1.
\end{align*}
(For noncompact varieties, the Deligne filtration is defined for the rational cohomology \emph{with compact support} \cite{Deligne}, but we do not need it here.)

Under the same completeness condition on $\cV$, Totaro \cite[Thm. 3]{TTL} has shown
\begin{equation}\label{operational}
\bigoplus_{a\ge0}\gr_{2a}^WH^{2a}(\cV,\QQ)=A^*(\cV)_\QQ,
\end{equation}
where $W$ refers to the $\NN$-weights with doubled indices (to make it compatible with Deligne's weights in the complete case) and $A^*$ is the Fulton-MacPherson operational Chow cohomology \cite{Fintersection}. In fact, \cite[Thm. 3]{TTL} establishes the dual rational isomorphism between the Chow groups and the smallest subspace of the Borel-Moore homology with respect to the weight filtration, the duality between $A_*$ and $A^*$ being provided by \cite[Thm. 3]{FMSSspheric}.

\subsection{Degeneration at the second page}
The following lemma extends the standard degeneration of $\tilde E_{\QQ}$ at the second page (e.g., \cite[Prop. 12.3.10]{CLStoric}) for the $H^*$-spectral sequence $\tilde E$, introduced in Section \ref{Singular}.

\begin{lemma}\label{singcohdeg}
For each $r\ge2$ there is an integer $n$ such that the differentials $d_r:\tilde E^{p,q}_r\to\tilde E^{p+r,q-r+1}_r$ satisfy $nd_r=0$ for all $p,q\ge0$. Hence the spectral sequence $\tilde E_{\QQ}$ degenerates at the second page. Moreover, the row $q=0$ satisfies $\tilde E^{p,0}_2=\tilde E^{p,0}_\infty$.
\end{lemma}

\begin{proof}
We write $d=d_r$. If $q<r-1$ and $z\in\tilde E^{p,q}_r$ then $dz=0$. Assume $q\ge r-1$. Then for all $c\in\NN$ we have
$$
c^{q-r+1}dz=c_*(dz)=d(c_*(z))=d(c^qz)=c^qdz,\quad\text{or}\quad c^{q-r+1}(c^{r-1}-1)dz=0.
$$
So we can choose $n=\gcd(c^{q-r+1}(c^{r-1}-1)\ |\ c\ge2,\ q\ge r-1)$.

Taking $c=2$ shows that the odd part of $n$ divides $2^{r-1}-1$, and taking $c=3$ shows that the 2-primary part of $n$ divides $3^{r-1}-1$. So $n\le(2^{r-1}-1)(3^{r-1}-1)$.
\end{proof}

The next proposition is Totaro's degeneration result for the $KH$-spectral sequence ${}'E$ (introduced in Section \ref{Cohomological}), after tensoring with $\QQ$. The argument is a variation of the degeneration of $\tilde E_{\QQ}$ at the second page and it uses the two splittings, provided by the $\NN$- and Adams weights.

\begin{theorem}\label{totarodegeneration}
Let $R$ be a regular ring and $\cF$ a quasi-projective fan. Assume that either $R$ contains a field or $\cF$ is simplicial. Then ${}'E_{\QQ}$ degenerates at the second page.
\end{theorem}

\begin{proof}
Let $\sigma_i$, $U_{\sigma_i}$, $M(\sigma)$ be as in Section \ref{Cohomological}. We can assume $n\ge3$, for otherwise there is nothing to prove.

First, we observe that the differential $d_1$ is an integer linear combination of restriction maps $KH_*(U)\to KH_*(U')$ for $U'\subset U$. Also, under the restriction map the subgroup $\Lambda^kM(\sigma)$ of $KH_n(U)$ maps to the subgroup $\Lambda^kM(\sigma')$ of $KH_n(U')$. Consequently, if we let $A$ be the direct sum of the images of the subgroups $\Lambda^*M=\Lambda^*M\otimes\ZZ$ of $\Lambda^*M\otimes K_0(R)$, where the larger groups are those showing up in the first page of ${}'E$, then $d_1$ maps $A$ into itself.

By (\ref{Epq}), the first page of ${}'E$ is $A\otimes K_*(R)$ and the second page is $H(A,d_1)\otimes K_*(R)$. By Lemma \ref{Adamsoperations}(b), to show that ${}'E_\QQ$ degenerates at the second page, it suffices to show that all differentials starting at $H(A,d_1)_{\QQ}$ are $0$.

Fix $c\ge2$. The groups of $\NN$-weight a power of $c$, showing up in the first page of ${}'E_{\QQ}$, have the form $\Lambda^cM_{\QQ}\otimes K_{q-c}(R)$ for some lattice $M$. By (\ref{Soule}), for $q>0$ we have: (i) the group $K_q(R)_{\QQ}$ splits into subgroups of positive Adams weights, and (ii) the subgroup $M_{\QQ}\subset KH_1(U)_{\QQ}$ (for an appropriate $U$) has Adams weight 1. Since the maps $\psi^k$ are multiplicative, the group $\Lambda^cM\otimes K_{q-c}(R)$ for $c<q$ has Adams weights $>c$.

Consider a differential $d_r:\big(\text{subquotient of}\ \Lambda^qM_\QQ\big)\to({'E}_r^{p+r,q+r-1})_\QQ$. The $\NN$- and Adams weights of the source are both $q$. (For the Adams weights one uses the third equality of (\ref{Soule}) and the multiplicative property of the $\psi^k$.) But for $r\ge2$, the previous paragraph shows that the part of the target group of $\NN$-weight $q$ has the form $\Lambda^qM'_{\QQ}\otimes K_{r-1}(R)_{\QQ}$ for some lattice $M'$. By the same paragraph, the latter group splits into subgroups of Adams weight $>q$. By Lemma \ref{Adamsoperations}(a) the weights must be respected, and this implies $d_r=0$.
\end{proof}

\begin{corollary}\label{totaroequalities}
Assume $\cF$ is a projective fan, $\cV=\cV_{\text{top}}(\cF)$, and $R$ is a regular ring.
\begin{enumerate}
\item[(a)]
The map $K_*(\cV_R(\cF))\to KH_*(\cV_R(\cF))$ is split surjective.
\item[(b)] If $R$ contains a field then
$$
KH_n(\cV_R(\cF))_{\QQ}=\left[\bigoplus_{p,q\ge0}\gr^W_{2(p+n-q)}H^{2p+n-q}(\cV,\QQ)\right]\otimes K_q(R).
$$
\item[(c)]
If $\cF$ is simplicial then, additively,
$$
KH_*(\cV_R(\cF))_\QQ=K_*(R)^m_\QQ,\quad m=\#\max(\cF).
$$
\end{enumerate}
\end{corollary}

\begin{remark}\label{Quasiprojective}
(a) Corollary \ref{totaroequalities}(b) and (\ref{operational}) imply that, if $\cF$ is projective and $R$ contains a field then $KH_0(\cV_R(\cF))_\QQ$ surjects onto $A^*(\cV)_\QQ\otimes K_0(R)$.

(b) If $\cF$ is just quasi-projective and $R$ contains a field, then Corollary \ref{totaroequalities}(b) is still valid in the form:
$$
KH_n(\cV_R(\cF))_{\QQ}=\left[\bigoplus_{p,q\ge0}\gr^W_{2(p+n-q)}H^{2p+n-q}\right]\otimes K_q(R),
$$
where the groups in the direct sum are associated with just the $\NN$-weight filtration of $H^*(\cV(\cF),\QQ)$, without the Delign filtration interpretation.
\end{remark}

\begin{proof}[Proof of Corollary \ref{totaroequalities}]
(a) As mentioned after (\ref{splitsurijective}), the splitting result is given in \cite[Prop. 5.6]{CHWWtoric}.

By Theorem \ref{totarodegeneration}, $KH_n(\cV_R(\cF))_\QQ$ is the direct sum of the graded pieces which align in $(E^{p,q}_2)_\QQ$ along the shifted diagonal $q=p+n$. Since the homomorphisms in (\ref{toric Mayer-Vietoris}) are direct sums of the homomorphisms in (\ref{singcoh}) with a certain shift (and tensored with the $K_j(R)$), we only need to keep track of the graded pieces from $H^*(\cV,\QQ)$, picked up in the summation process. This is most conveniently done by grouping w.r.t to the tensor factors $K_q(R)$.

\medskip\noindent(b) By Oda's theorem \cite[Thm. 12.3.11]{CLStoric}, $H^{2p}(\cV,\QQ)$ is \emph{pure}, i.e., $H^{2p}(\cV,\QQ)=\gr^W_{2p}H^{2p}(\cV,\QQ)$. (As remarked above, the cohomology spectral sequence, used in \cite[Ch. 12]{CLStoric}, is different from $\tilde E$.) On the other hand, we have the following Betti number counting (\cite[Thm. 12.3.12]{CLStoric}):
$$
b_{2p}(\cV)=\dim H^{2p}(\cV,\QQ)=\sum_{i=p}^{\dim\cF}(-1)^{i-p}{i\choose p}\#\cF(d-i),\quad p\ge0,
$$
where $\cF(d-i)$ is the set of $(d-i)$-dimensional cones in $\cF$. Consequently,
$$
b_0(\cV)+b_2(\cV)+b_4(\cV)+\cdots=\#\cF(d)=m.
$$
So, by the equality in (a), we have $KH_*(\cV_R(\cF))_\QQ=K_*(R)_\QQ^m$.
\end{proof}

\noindent\emph{Notice.} Most likely, Corollary \ref{totaroequalities}(c) extends to all complete -- not necessarily projective -- simplicial fans $\cF$: the same argument goes through word-for-word if Adams operations commute with the corresponding Mayer-Vietoris differentials.

\section{Nil-groups}\label{Conjecture}

For an affine monoid $M$ the \emph{normalization} and \emph{seminormalization} of $M$ are defined as follows
\begin{align*}
&\n(M)=\{x\in\gp(M)\ |\ cx\in M\ \text{for some}\ c\in\NN\},\\
&\sn(M)=\{x\in\gp(M)\ |\ cx\in M\ \text{for all sufficiently large}\ c\in\NN\}.
\end{align*}
They are, correspondingly, the smallest seminormal and normal monoids in $\gp(M)$, containing $M$.

For a finitely generated cone $C$, we denote by $\FF(C)$ the set of nonzero (i.e., positive dimensional) faces of $C$ and by $\inte(C)$ the relative interior of $C$ in the ambient Euclidean space.

\begin{conjecture}\label{conjecture}
Let $n$ be a nonnegative integer, $R$ a regular ring with $\QQ\subset R$, $M$ a positive affine monoid, and $C=\RR_+M$. Denote $\tilde K_n(R[M])=K_nR[M])/K_n(R)$.
Then there exists an $\sn(M)$-grading
$$
\tilde K_n(R[M])=\bigoplus_{\sn(M)\setminus\{0\}}A_j,
$$
satisfying the conditions:
\begin{enumerate}
\item[(a)] For every face $F\in\FF(C)$ there exist $m_F\in\inte(F)\cap\sn(M)$ such  that
$$
\Supp(\tilde K_n(R[M]))\subset\sn(M)\setminus\bigcup_{\FF(C)}(m_F+F);
$$
\item[(b)] $\Supp(c_*(x))\subset c\Supp(x)$ for all $x\in\tilde K_n(R[M])$ and $c\in\NN$;
\item[(c)] There is an $R[M]$-module structure on $\tilde K_n(R[M])$ with $c_*(mx)=m^cc_*(x)$,
\item[(d)] As an $R[M]$-module, $\tilde K_n(R[M])$ is generated by $\bigoplus_JA_j$ for a finite subset $J\subset\sn(M)$. Moreover, when $R$ is a number field, $\tilde K_i(R[M])$ is a finitely generated $R[M]$-module.
\end{enumerate}
\end{conjecture}

The conjecture is true for $i=0$. In fact, by \cite[Thm. 8.42]{kripo},
$$
\Pic(R[M])/\Pic(R)=R[\sn(M)]/R[M]=R(\sn(M)\setminus M),
$$
where the rightmost entry is the free $R$ module on $\sn(M)\setminus M$. Actually, this is the monoid-ring case of \cite{Roberts}: the latter proves the same equality for Picard groups of graded rings. However, the `monoid friendly' argument, presented in \cite{kripo}, makes clear that the $c_*$-action on $\Pic(R[M])/\Pic(R)$ is exactly the Frobenius $c$-homothety on the $R$-module quotient $R[\sn(M)]/R[M]$. Moreover, the second equality in (\ref{K0}) implies $\tilde K_0(R[M])=\Pic(R[M])/\Pic(R)$. Further, the $R$-module structure results from the Stienstra-Weibel $\W(R)$-action (upon fixing a grading on $R[M]$) and the fact that $R\subset\W(R)$ -- a consequence of the inclusion $\QQ\subset R$. This shows (b,c,d), and (a) follows from the fact that for any affine positive monoid $L$ there exists $l\in\inte(\RR_+L)\cap L$ such that $l+\n(L)\subset L$ \cite[Prop. 2.33]{kripo}.

\medskip\noindent\emph{Notice.} The first step in proving Conjecture \ref{conjecture} could be a multivariate version of \cite{Sop}. Observe, Conjecture \ref{conjecture}(a,b) implies the existence of a natural number $c(R,M,n)$, such that $c_*$ annihilates all of $\tilde K_n(R[M])$ as soon as $c\ge c(R,M,n)$. One should also remark that the elements $m_F$ in Conjecture \ref{conjecture}(a) depend on $n$, as illustrated by Geller-Reid-Weibel's $K$-theoretical computations for cusps $k[t^u,t^v]$, $\gcd(u,v)=1$ (\cite[Thm. 9.2]{Curves}).

\medskip\noindent\emph{Acknowledgement.} Thanks are due to Wojciech Gajda for showing me Totaro's notes in Bielefeld many years ago, Florence Lecomte for comments on Adams operations, and Chuck Weibel and the referee for still more comments and spotting inaccuracies.

\bibliography{references}
\bibliographystyle{plain}

\end{document}